\documentclass[12pt]{article}   % If the bibliography is changed:
\usepackage{graphicx,psfrag}
\usepackage[table]{xcolor}
\usepackage{color}
\usepackage{tikz}
\usepackage{pgfplots}
\usepackage{caption}
\usepackage{subcaption}
\usepackage{array}
\usepackage{colortbl}
\usepackage{amsmath,amsfonts}
\usepackage{amssymb,amsthm}
\usepackage{subcaption}
\usepackage{epstopdf}

\usepackage{tabularx}
\newtheorem{theorem}{Theorem}[section]
\newtheorem{corollary}[theorem]{Corollary}
\newtheorem{proposition}[theorem]{Proposition}

\newtheorem{example}[theorem]{Example}
\newtheorem{lemma}[theorem]{Lemma}

\newtheorem{remark}[theorem]{Remark}

\def\qed{\vbox{\hrule
 \hbox{\vrule\hbox to 5pt{\vbox to 8pt{\vfil}\hfil}\vrule}\hrule}}

\def\endproof{\unskip \nobreak \hskip0pt plus 1fill \qquad \qed \par}

\newcommand{\mb}{\mathbb}
\newcommand{\mc}{\mathcal}

\usepackage{tikz}
\usetikzlibrary{decorations.markings}

\begin{document}

\title{Sign-restricted matrices of $0$'s, $1$'s, and $-1$'s}

\author{
 Richard A. Brualdi\footnote{Department of Mathematics, University of Wisconsin, Madison, WI 53706, USA. {\tt brualdi@math.wisc.edu}} \\
 Geir Dahl\footnote{Department of Mathematics,  
 University of Oslo, Norway.
 {\tt geird@math.uio.no.} Corresponding author.}}
\date{2 January 2021}
\maketitle

\begin{abstract}  
  We study {\em sign-restricted matrices} (SRMs), a class of rectangular $(0, \pm 1)$-matrices generalizing the  alternating sign matrices (ASMs). In an SRM  each partial column sum, starting from row 1, equals 0 or 1, and each partial row sum, starting from column 1, is nonnegative. We determine the maximum number of nonzeros in SRMs and  characterize the possible row and column sum vectors. Moreover, a number of results on interchange operations are shown, both for SRMs and, more generally, for $(0, \pm 1)$-matrices. The Bruhat order on ASMs can be extended to SRMs with the result a distributive lattice. Also,  we study polytopes associated with SRMs and some relates decompositions. 
 \end{abstract}

%\end{frontmatter}

\noindent {\bf Key words.} Sign matrix, alternating sign matrix, sign-restricted matrix, polytope, Bruhat order.

\noindent
{\bf AMS subject classifications.} 05A18, 05B20, 06A07, 15B35, 15B36.

\section{Introduction}
 \label{sec:intro}
 
 Let $m$ and $n$ be positive integers and let ${ \Delta}_{m,n}$ be the set of all $m\times n$ matrices each of whose entries is $0$, $+1$, or $-1$, that is, $(0,\pm 1)$-matrices. Perhaps the best known class of $(0,\pm 1)$-matrices are the $n\times n$ {\it alternating sign matrices} (ASMs) \cite{BN,BD1,BD2,DB,St}. These
  are square matrices in which the $\pm 1$'s in each row and column alternate beginning and ending with a $+1$, and hence for which all row and column sums equal 1. The set of $n\times n$ ASMs is denoted by ${\mathcal A}_n$. In \cite{JA,SS} a generalization of ASMs, called sign matrices, has been defined and these matrices can be  rectangular. We prefer to call these matrices ``sign-restricted matrices'' to emphasize the restrictions on the signs, and they are defined next.

  A {\em sign-restricted matrix} (abbreviated here to SRM)  is  an $m\times n$ $(0,\pm 1)$-matrix $A$ such that each partial column sum, starting from row 1, equals 0 or 1, and each partial row sum, starting from column 1, is nonnegative.  This definition arose in   \cite{JA}  where they are shown to be in bijection with combinatorial objects called semistandard Young tableaux, and they were further investigated in \cite{SS}.   Row 1 and column 1 of an SRM can only contain 0's and $+1$'s and, in particular, column 1 can contain only one $+1$.
 Also the $+1$'s and $-1$'s in a column alternate, where the corresponding full column sum is $1$ or $0$ depending on whether its last nonzero entry, if any,  is a 1 or a $-1$.  Note that the transpose of an SRM need not be an SRM.  If the rows of $A$
satisfied the same property as the columns and $A$ were a square
matrix, then $A$ would be an ASM.   Any leading $r\times s$ submatrix of an ASM is an SRM.\footnote{Of course, the rows of a leading submatrix of an ASM satisfy the same properties as its columns, and this is not reflected in the definition of a sign-restricted matrix, since a weaker condition, namely partial row sums are non-negative, is used.}  
   Unlike ASMs, the last nonzero entry in each row and column of an SRM may be a $-1$. 
 We denote the set of $m\times n$ SRMs  by ${\mathcal S}_{m,n}$. The subset of ${\mathcal S}_{m,n}$ consisting of those matrices with no $-1$'s (so $(0,1)$-matrices) is denoted by  ${\mathcal S}_{m,n}^+$.

  A $1\times n$ SRM is just a $(0,1)$-vector. Since an $m\times n$ SRM has a column sum vector which is a $(0,1)$-vector, an SRM can be considered as a generalization of a $(0,1)$-vector. 
 %By including additional rows with a single $1$ and no $-1$, one could assume that all column sums equal 1. 
 
 \begin{example}\label{ex:onea}{\rm Examples of SRMs are
 \[\left[
 \begin{array}{rrr}
 0&1&1\\
 1&-1&0\\
 0&1&-1\end{array}\right],
 \left[\begin{array}{rrrr}
 0&1&0&1\\
 1&-1&1&-1\\
 0&1&-1&1\end{array}\right], \mbox{ and }
 \left[\begin{array}{rrr}
 0&1&1\\ 1&-1&0\\0&1&-1\\ 0&0&1\end{array}\right] .\]
  A zero matrix is  a sign-restricted matrix as is every permutation matrix and subpermutation matrix. 
 }\hfill{$\Box$}\end{example}

 Let $A$ be an $m\times n$ SRM with row sum vector $R=(r_1,r_2,\ldots,r_m)$ and column sum vector $S=(s_1,s_2,\ldots,s_n)$. As remarked above, $S$ is a $(0,1)$-vector, but $R$ may have integer entries larger than 1. If each partial row sum of $A$ equals $0$ or $1$, then $R$ is also a $(0,1)$-vector and  the transpose of $A$ is also a sign-restricted matrix.  In general, the number of $1$'s in $S$ equals $\sum_{i=1}^m r_i$.  For instance, the row sum vectors of the SRMs  in Example \ref{ex:onea} are $(2,0,0)$, $(2,0,1)$, and $(2,0,0,1)$, respectively. The column sum vectors are $(1,1,0)$, $(1,1,0,1)$, and $(1,1,1)$, respectively. The   set of all SRMs with row sum vector $R$ and column sum vector $S$ is denoted by ${\mathcal S}(R,S)$, or by ${\mathcal S}_{m,n}(R,S)$ if we want to emphasize the dimensions of $R$ and $S$. Similarly we use the notations ${\mathcal S}^+(R,S)$ and ${\mathcal S}_{m,n}^+(R,S)$ to denote the SRMs with nonnegative entries
 
 Notation: We let $M_{m,n}$ denote the set of real $m \times n$ matrices, and simply write $M_n$ when $m=n$. 
 
The remaining paper is organized as follows. Section \ref{sec:basic-prop}
considers the maximum number of nonzeros in SRMs and  characterizes the possible row and column sum vectors of SRMs. Also, we study a connection to another class of $(0, \pm 1)$-matrices containing the incidence matrices of directed graphs. In Section \ref{sec:interchangeSRM} we consider the class of SRMs with specified row and column sum vectors, and their connectivity properties under interchanges are investigated.  Section \ref{sec:Bruhat} is devoted to the Bruhat order for  the class $\mc{S}_{m,n}$;  we show that this determines a distributive lattice, and it is the Dedekind-MacNeille completion of the Bruhat order restricted to  $\mc{S}^+_{m,n}$.  In Section \ref{sec:polytope} we study a polytope associated with SRMs and some relates decompositions.

 \section{Some basic properties of SRMs}
 \label{sec:basic-prop}
 
 First we consider the row and column sum vectors of an SRM.

\begin{proposition}
 \label{pr:row-sums1}
  Let $R=(r_1, r_2, \ldots, r_m)$ and $S=(s_1, s_2, \ldots, s_n)$ be nonnegative integral vectors, respectively. Then $R$ and $S$ are the row sum and column sum vectors, respectively,  of an SRM $A$ of size $m \times n$ if and only if $S$ is a $(0,1)$-vector and 
  \begin{equation}
   \label{eq:RS-cond}
    \sum_{i=1}^m r_i=\sum_{j=1}^n  s_j.
  \end{equation}
  In fact, $A$ can be taken to be a matrix in ${\mathcal S}_{m,n}^+$.
\end{proposition}
\begin{proof}
 The conditions are clearly necessary, as an SRM has column sums 0 or 1, and (\ref{eq:RS-cond}) is trivial. Conversely, assume $S$ is a $(0,1)$-vector and (\ref{eq:RS-cond}) holds. Let $k= \sum_{i=1}^m r_i=\sum_{j=1}^n  s_j$. Then $k \le n$ as $S$ is a $(0,1)$-vector. 
Initially, let $A =[a_{ij}]$ be the $(0,1)$-matrix with a 1 in the  first row in those columns for which $s_j=1$, while all other entries are zero. Thus $A$ has column sum vector $S$. Next, we modify $A$ by shifting the ones in the first row to other rows so as to obtain the row sum vector $R$; this may be done as $\sum_{i=1}^m r_i=k$. 
For instance, this may be done so that the ones are in a ``staircase'' pattern in the sense that whenever $a_{ij}=a_{kl}=1$ and $j<l$, then $i\le k$. The resulting matrix is a $(0,1)$-matrix, so a matrix in ${\mathcal S}_{m,n}^+$.
\end{proof}

 For $m$ and $n$ fixed, the maximum of  $\sum_{i=1}^n r_i$  in an $m\times n$ SRM is $n$. This follows from Proposition \ref{pr:row-sums1} as each column sum is 0 or 1 so that $\sum_i r_i=\sum_j s_j \le n$. This bound is attained by taking $S$ be the all ones vector and $R=(n, 0, \ldots, 0)$.

 The first column of an $m\times n$ SRM can contain only one nonzero and this nonzero is a 1. The second column can then contain at most three nonzeros, the third column at most five nonzeros, and so on,  until we reach column $\lceil\frac{m+1}{2}\rceil$ which can contain 
 $m$ nonzeros. After that we can alternate between columns containing $(m-1)$ and $m$ nonzeros. This construction gives an $m\times n$ SRM with the maximum number $\zeta_{m,n}$ of nonzeros.
 
 \begin{example}{\rm \label{ex:max}
 We  give two examples to illustrate how the maximum is obtained.
 \[(m=6,n=8): \left[\begin{array}{r|r|r|r|r|r|r|r}
 &&&1&&1&&1\\ \hline
 &&1&-1&1&-1&1&-1\\ \hline
 &1&-1&1&-1&1&-1&1\\ \hline
 1&-1&1&-1&1&-1&1&-1\\ \hline
 &1&-1&1&-1&1&-1&1\\ \hline
 &&1&-1&1&-1&1&-1\end{array}\right]\; \;  (\zeta_{6,8}=37),\]
 whose  column sum vector is $(1,1,1,0,1,0,1,0)$,
 and
 
 \[(m=9,n=11): \left[\begin{array}{r|r|r|r|r|r|r|r|r|r|r}
 &&&&1&&1&&1&&1\\ \hline
 &&&1&-1&1&-1&1&-1&1&-1\\ \hline
 &&1&-1&1&-1&1&-1&1&-1&1\\ \hline
 &1&-1&1&-1&1&-1&1&-1&1&-1\\ \hline
 1&-1&1&-1&1&-1&1&-1&1&-1&1\\ \hline
 &1&-1&1&-1&1&-1&1&-1&1&-1\\ \hline
 &&1&-1&1&-1&1&-1&1&-1&1\\ \hline
 &&&1&-1&1&-1&1&-1&1&-1\\ \hline
 &&&&1&-1&1&-1&1&-1&1
 \end{array}\right]\;  (\zeta_{9,11}=76).\]
 }\hfill{$\Box$}
 \end{example}
 
 The construction used in Example \ref{ex:max} can be used to give a formula for $\zeta_{m,n}$ which, 
 in the case of $m=n$, is very simple.

 \begin{theorem}\label{th:max} Let $m$ and $n$ be  positive integers. Then
 \begin{equation}
 \label{eq:max-mn}
 \zeta_{m,n}=
 \left\{
 \begin{array}{ll}
 mn-m^2/4-\lceil(1/2)(n-m/2-1)\rceil, \quad &\mbox{ if $m$ is even},\\
 \\
 mn-(m^2-1)/4-\lceil(1/2)(n-(m+1)/2)\rceil , \quad &\mbox{ if $m$ is odd}.\\
 \end{array}
 \right.
 \end{equation}
 
In particular, if $m=n$, 
 \begin{equation}
 \label{eq:max}
 \zeta_{n,n}=
 \left\{
 \begin{array}{ll}
 \frac{3n^2-n}{4}, \quad &\mbox{ if $n\equiv 0\mbox{ or }3\bmod 4$},\\
 \\
  \frac{3n^2-n+2}{4}, \quad &\mbox{ if $n\equiv 1\mbox{ or } 2\bmod 4$}.\\
%  \\
  %\frac{3n^2-n+2}{4}, \quad \mbox{ if $n\equiv 1\bmod 4$},\\
 % \\
% \frac{3n^2-n}{4}, \quad \mbox{ if $n\equiv 3\bmod 4$}.
 \end{array}
 \right.
 \end{equation}
 \end{theorem}
 
 \begin{proof} This is 
 a straightforward computation using the above construction.
 \end{proof}

The maximum difference between the number of 1's and the number of $-1$'s in an $m\times n$ SRM  is $n$. In fact, in each column, due to the alternating property,  the difference between the number of 1's and the number of $-1$'s is either 0 or 1. Thus, for a matrix in $\mc{S}_{m,n}$, the maximum difference between the number of 1's and the number of $-1$'s is $n$, and this is attained for the  matrix whose first row is the all ones vector, and all other entries are 0.

%  \begin{question}
%  \label{qu:two} 
%  {\rm 
%  We can consider other questions concerning SRMs with given row  and column sum vectors $R$ and $S$. 
 
% \begin{itemize}
% \item[\rm (i)] Which $R$ and $S$ can be obtained for SRMs with at least one negative entry?
% \item[\rm (ii)] For fixed $R$ and $S$ how many negative entries can a matrix in $\mc{S}(R,S)$ have? 
% \end{itemize}
% }
% \end{question} \endproof

  The incidence matrices associated with directed graphs give a well-known class of $(0, \pm 1)$-matrices. Let $D=(V,E)$ be a directed graph with at least one (directed) edge and vertices $\{1,2,\ldots,n\}$. An edge from a  vertex $i$ to a vertex $j$ is denoted by $(i,j)$; we assume that $D$ does not have any loops so that
  $i\ne j$.  The incidence matrix $M$ of $D$ then has rows corresponding to its vertices (in some order) and columns corresponding to  its edges (again, in some order). The column corresponding to the edge $(i,j)$ has a 1 in row $i$, a $-1$ in row $j$, and otherwise only zeros. In particular, each column contains exactly two nonzeros. The first column of $M$ contains a $-1$ no matter how the edges of $D$ are ordered, and so some row will begin with a $-1$, and  hence $M$ is never an SRM. We can remedy this by using loops.

   Let a loop of a digraph at a vertex $i$  correspond to a column in the incidence matrix with a 1 in row $i$ and otherwise all 0's.\footnote{This
 is justified as the loop meets only the vertex $i$ and it allows one to identify from the incidence matrix which vertices have a loop.} 
 %(*** {\bf DELETE THIS?}: But note that sometimes loops correspond to a zero column which therefore does not allow their identification.}  
 So putting  loops at all vertices and letting these loops correspond to the first columns of the incidence matrix (so the incidence matrix begins with the identity matrix $I_n$), then no row will begin with a $-1$.  But this does not guarantee that the incidence matrix is an SRM under some ordering of the other edges. We now determine when including certain loops leads to an SRM. 
   
   Let $S\subseteq V$ and let $D(S)$ be the digraph obtained from $D$ by putting a loop at each vertex in $S$. Let $M(S)$
   be its  incidence matrix obtained from the incidence matrix $M$ of  $D$ by  augmenting $M$ by distinct unit vectors corresponding to vertices in $S$ {\it where these unit vectors come first}.  We call $M(S)$ the {\em generalized incidence matrix} for $D=(V,E)$ and $S$. Let $d^+(v)$ (resp. $d^-(v)$) denote the number of edges of $D$ with $v$ as tail (resp. head).

\begin{example}{\rm 
\label{ex:inc-matrix}
 Let $D$ be the digraph with vertices $v_i$ ($i \le 4$) and edges $(v_1,v_2)$, $(v_2,v_3)$ and $(v_2,v_4)$. Let $S=\{v_3,v_4\}$. The generalized incidence matrix, using vertex order according to their index and edge order $(v_2,v_3)$, $(v_1,v_2)$, $(v_2,v_4)$ is 
 \[
 \left[
 \begin{array}{rrrrr}
 0&0&0&1&0\\
 0&0&1&-1&1\\
 1&0&-1&0&0\\
 0&1&0&0&-1
 \end{array}\right].
 \]
 This is an SRM. \endproof
 } 
 \end{example} 
  
\begin{theorem}
 \label{thm:incidence-matrix}
  Let $M(S)$ be the generalized incidence matrix associated with $D=(V,E)$ and $S \subseteq V$ where $D$ is a directed graph with at least one edge.  Then  the rows and columns of $M$ can be reordered so that the resulting matrix is an SRM if and only if the following holds
  \begin{itemize}
   \item[\rm (i)] $D$ is acyclic, 
  
   \item[\rm (ii)]  $d^-(v)-d^+(v) \le 1$ for each vertex $v$, and  if $d^-(v)=d^+(v)+1$ for some vertex $v$, then $v \in S$.
  \end{itemize}
\end{theorem}
\begin{proof}
 First assume that the rows and columns of $M$ are ordered so that $M(S)$ is an SRM. Suppose to the contrary that $D$ contains a directed cycle $C$, and let $i$ be the first row in $M$ such that the corresponding vertex lies in $C$. Moreover, $C$ contains exactly two edges incident to vertex $i$, one with $i$ as its head, and one with $i$ as its tail. For each of these two edges, the other end vertex (different from $i$) corresponds to some row below row $i$. Therefore, since each column of $M$ contains exactly two nonzero entries, one of the columns of $M$ (and so of $M(S)$)  corresponding to these two edges must have its first nonzero equal to $-1$; contradicting that $M(S)$ is a SRM. This shows that $D$ must be acyclic. 
 
  Next, assume there is a vertex $v$ of $D$ with $d^-(v)-d^+(v) \ge 2$, i.e., the indegree is at least 2 larger than the outdegree. But then the row sum of $M$ in that row is at most $-2$ and hence the corresponding row sum of $M(S)$ is negative. Therefore $d^-(v)-d^+(v) \le 1$ for each vertex $v$. Moreover,  if $d^-(v)=d^+(v)+1$ for some vertex $v$, then $v \in S$; otherwise the row sum would be negative. So, (ii) holds.
  
  Conversely, assume conditions (i)--(ii) hold. As $D$ is acyclic, its vertices, and the rows of $M$,  may be ordered $v_1, v_2, \ldots, v_n$ such that each edge has the form $(v_i,v_j)$ for some $i<j$. 
Next we describe a suitable ordering of the columns of $M$ so that $M(S)$ is a SRM.    Choose $v_j$ with $d^+(v_j)=0$, $d^-(v_j)=1$ and $j$ maximal with this property; such a vertex must exist. Choose a directed path $P$ in $D$ with a maximal number of edges and with terminal vertex $v_j$. Order the edges consecutively along the path $P$  with the edge having $v_j$ as  its head as the first one. Then order the  corresponding columns of $M$ in a similar order. In the submatrix defined by the columns corresponding to $P$ each row contains a $1$ and a $-1$, in that order, except the row corresponding to $v_j$ where the only nonzero is a $-1$. However, that row has a $1$ in an earlier column, as  unit elements in $S$ are  first. 
Now, remove the edges of $P$  from $D$ and repeat this procedure in the resulting graph, by choosing such a path and order the corresponding columns accordingly. Then the resulting matrix is an SRM, as desired.  
 \end{proof}
 
Example \ref{ex:inc-matrix} illustrates Theorem \ref{thm:incidence-matrix} and the vertex and edge orders used in the proof.  
In summary, the theorem asserts that if a digraph $D$ satisfies (i) and (ii)  one may insert as the initial columns in its incidence matrix a set of distinct unit columns so that the resulting matrix in a SRM. Moreover, the unit columns needed are identified.

 \section{Interchanges}
 \label{sec:interchangeSRM}
 
 In this section we first consider certain connectivity properties of the class $\mc{S}_{m,n}$. Define  
 \begin{equation}
 \label{eq:E}
 E=\left[
  \begin{array}{rr}
     1 & -1 \\
     -1 & 1  
  \end{array}
 \right].
 \end{equation}
 Let $A$ be an SRM and let $B$ be obtained from $A$ by adding or subtracting $E$ in some $2 \times 2$ submatrix of $A$ (not necessarily with consecutive rows and consecutive columns). We call this operation an {\em interchange}. Whether or not  $B$ is an SRM depends on $A$ and the chosen submatrix. In any case $R(A)=R(B)$ and $S(A)=S(B)$, where $R(A)$ (resp., $S(A)$) is the  row sum (resp., column sum) vector of $A$, and similarly for $B$.
 
 We now establish the following interchange result, the first conclusion of which shows that by a sequence of interchanges every matrix in ${\mathcal S}_{m,n}(R,S)$ can be brought to a matrix in ${\mathcal S}_{m,n}^+(R,S)$.
 The result is related to  the construction in \cite{JA} of the ``key" of an ASM.

 \begin{theorem}
 \label{th:interchange}
  Let $A, B \in \mc{S}_{m,n}$  with row sum vectors   $R=R(A)=R(B)$ and column sum vectors $S=S(A)=S(B)$. 
  \begin{itemize}
  \item[\rm (i)] There exist SRMs $A^{(i)}$ $(1\le i \le k)$ such that $A^{(1)}=A$ and  $A^{(k)}\in \mc{A}(R,S)$, and $A^{(i+1)}$ is obtained from $A^{(i)}$ by an interchange $(1\le i < k)$.   
  \item[\rm (ii)]
  There exist SRMs $A^{(i)}$ $(1\le i \le p)$ such that $A^{(1)}=A$, $A^{(p)}=B$, and $A^{(i+1)}$ is obtained from $A^{(i)}$ by an interchange $(1\le i < p)$. 
  \end{itemize}
 \end{theorem}
 \begin{proof}
  Assume first that $A=[a_{ij}]$ has at least one $-1$. Choose a position $(i,j)$ with $a_{ij}=-1$ and $i+j$ minimal with this property; we then call $(i,j)$ a {\em top-left} position of a $-1$. (Such a position may not be unique, but this has no importance.) Since $A$ is an SRM there must exist a $k<i$ such that $a_{kj}=1$ and there exists $l<j$ such that $a_{il}=1$ (as the first nonzero in a row or column cannot be a $-1$). Now, we must have $a_{kl}=0$. In fact, $a_{kl}$ cannot be $-1$ as $(i,j)$ is a  top-left position of a $-1$. Moreover, $a_{kl}$ cannot be $1$, because then column $l$ would have to contain a $-1$ in some position $(i',l)$ with $k<i'<i$, again contradicting that $(i,j)$ a top-left position of a $-1$. 
  
  Let the matrix $A'$ be obtained from $A$ by adding the $2 \times 2$ matrix $E$ (see (\ref{eq:E})) to the submatrix of $A$ corresponding to rows $k, i$ and columns $l,j$. Then, from the properties just mentioned, it follows that $A'$ is  an SRM. Moreover, the number of $-1$'s in $A'$ is one less than the number in $A$. We can therefore repeat this procedure of interchanges and find a sequence of SRMs, each obtained by an interchange applied to the previous one, such that the final matrix $A^*$ does not have any entries equal to $-1$, i.e., it is a $(0,1)$-matrix. As mentioned, interchanges do not change any row or column sums, so $R(A^*)=R(A)$ and $S(A^*)=S(A)$. This proves (i).
  
Similarly, we may find interchanges and intermediary SRMs connecting $B$ to a   $(0,1)$-matrix $B^*$. Then $R(B^*)=R(B)=R(A)=R(A^*)$ and $S(B^*)=S(B)=S(A)=S(A^*)$, so $A^*$ and $B^*$ are both contained in ${\mc A}(R,S)$, and by the Ryser interchange theorem (see e.g.,  \cite{BR91}), one can use interchanges to go from $A^*$ to $B^*$ such that intermediary matrices are $(0,1)$-matrices. In fact, the last interchange result is easy to show directly, because $A^*$ and $B^*$ are $(0,1)$-SRMs, so each column contains at most one 1. The theorem now follows. 
 \end{proof}
 
 The algorithm given in the previous proof is illustrated in the next example.

  \begin{example}
 \label{ex:one}
 {\rm Consider the SRM $A$ below, and the transformation to an SRM which is a $(0,1)$-matrix:
 \[
 A=
 \left[\begin{array}{rrr} 0&1&1\\ 1&-1&0\\0&1&-1\\ 0&0&1\end{array}\right] 
 \rightarrow 
  \left[\begin{array}{rrr} \mbox{ }1&\mbox{ }0&1\\ 0&0&0\\0&1&-1\\ 0&0&1\end{array}\right] 
 \rightarrow 
 \left[\begin{array}{rrr} \mbox{ }1&\mbox{ }1&\mbox{ }0\\ 0&0&0\\0&0&0\\ 0&0&1\end{array}\right].
 \]
 Here we first added the submatrix $E$ to the submatrix given by the first two rows and columns, and then we added $E$ to the submatrix given by rows $1,3$ and columns $2,3$ to get the final matrix.
 } \endproof
\end{example} 

From the proof of Theorem \ref{th:interchange} it follows that every  matrix class $\mc{S}(R,S)$, consisting of SRMs with row sum vector $R$ and column sum vector $S$, contains a unique $(0,1)$-matrix $A=\bar{A}(R,S)$ with the following structure: ignoring zero columns (where $s_j=0$) the first row has ones in the first $r_1$ columns, the second has ones in the next $r_2$ columns etc. The example above shows the canonical matrix when $R=(2,0,0,1)$ and $S=(1,1,1)$. 
%This suggests the following   questions in the remark below.

We now turn to interchange properties of general $(0, \pm 1)$-matrices with  a prescribed row sum vector $R$ and column sum vectors $S$. Here again an interchange is adding or subtracting the matrix $E$ in (\ref{eq:E}) to some $2 \times 2$  submatrix in such a way that one obtains a new  $(0,\pm 1)$-matrix, necessarily with the same row sum vector $R$ and column sum vector $S$.
 
 Let $R=(r_1,r_2,\ldots, r_m)$ and $S=(s_1,s_2,\ldots,s_n)$ be nonnegative integral vectors with $r_1+r_2+\cdots+r_m=s_1+s_2+\cdots+s_n$. Let ${\mathcal A}(R,S)$ denote the class of $(0,1)$-matrices with specified row sum vector $R$ and specified column sum vector $S$. Also, let ${\mathcal A}^{\pm}(R,S)$ be the set of all $(0,\pm 1)$-matrices with row sum vector $R$ and column sum vector $S$. Let $J$ be the $m\times n$ matrix of all 1's. Then the mapping $A\rightarrow A+J$ is a bijection between ${\mathcal A}^{\pm}(R,S)$ and the set ${\mathcal A}^{0,1,2}(R',S')$ of all $m\times n$ $(0,1,2)$-matrices with row sum vector $R'=R+e_{(m)}$ and column sum vector $S'=S+e_{(n)}$ where $e_{(m)}$ (resp. $e_{(n)}$) is the all 1's vector of size $m$ (resp. size $n$).
 The special case of  Theorem 6.2.4 in \cite{BR91} obtained by taking $p=2$ gives a necessary and sufficient condition for the nonemptiness of ${\mathcal A}^{0,1,2}(R',S')$ and thus of ${\mathcal A}^{\pm}(R,S)$.
 Without loss of generality, $S$ can be assumed to be non-increasing.
 
 \begin{lemma}{\rm (\cite{BR91})}
 \label{lem:nonempty}
 Let $R=(r_1,r_2,\ldots, r_m)$ and $S=(s_1,s_2,\ldots,s_n)$ be nonnegative integral vectors with $r_1+r_2+\cdots+r_m=s_1+s_2+\cdots+s_n$. Assume that $S$ is nonincreasing. Then  ${\mathcal A}^{0,1,2}(R,S)\ne \emptyset$ if and only if
 \[\sum_{j=1}^k s_j\le\sum_{i=1}^m \min\{r_i,2k\},\ (k=1,2,\ldots, n).\]
 \end{lemma}

By the bijection of the previous paragraph, we
have the following as a corollary of Lemma \ref{lem:nonempty}.

\begin{corollary}
 ${\mathcal A}^{\pm}(R,S)\ne \emptyset$ if and only if
 \[\sum_{j=1}^ks_j\le \sum_{i=1}^m \min\{r_i+n,2k\}-km\ (k=1,2,\ldots, n).\]
\end{corollary} 
 
\begin{example}
{\rm 
Let $m=n=2$, and $R=S=(2,0)$. Then $\mc{A}^{\pm}(R,S)$ is nonempty and contains the matrix
\[
A=\left[
\begin{array}{rr} 
1& 1\\ 
1&-1
\end{array}
\right].
\]
The condition in the theorem becomes: 
\[
\begin{array}{ll}
   k=1: & 2=s_1 \le \min\{2,4\}+\min\{2,3\}-1\cdot 2=2 \\
   k=2: & 2+0=s_1+s_2 \le \min\{4,4\}+\min\{4,2\}-2\cdot 2=2.
\end{array}   
\]
Note that  in this example $\mc{A}(R,S)$ is empty.\endproof
}
\end{example}

 In \cite{Anstee83} (see also Theorem 4.4.6 and the paragraph following its proof in \cite {RAB}), the following result is established.
 
 \begin{lemma}\label{lem:anstee} Any two matrices in ${\mathcal A}^{0,1,2}(R',S')$ can be obtained from one another by a sequence of interchanges with all intermediary matrices also in ${\mathcal A}^{0,1,2}(R',S')$.
 \end{lemma}
 
 As an immediate corollary we obtain the following.
 
 \begin{corollary}\label{cor:ansteex}
 Any two matrices in ${\mathcal A}^{\pm } (R,S)$ can be obtained from one another by a sequence of interchanges with all intermediary matrices also in ${\mathcal A}^{\pm }(R,S)$.
 \end{corollary}

Note that in order to get a matrix with a $-1$ from a matrix $A$ in ${\mathcal A}(R,S)$, $A$ needs to have a $2\times 2$ submatrix with at most one 1.
 For instance,
 \[\left[\begin{array}{ccc}
 1&0&0\\ 0&1&0\\ 1&0&1\end{array}\right]+
 \left[\begin{array}{rrr}
 0&1&-1\\ 
 0&-1&1\\ 
 0&0&0\end{array}\right]=
 \left[\begin{array}{rrr}
 1&1&-1\\ 0&0&1\\ 1&0&1\end{array}\right].\]

 It follows from Corollary \ref{cor:ansteex} that if a nonempty class ${\mathcal A}^{\pm } (R,S)$ contains a $(0,1)$-matrix $A$, that is, a matrix in ${\mathcal A}(R,S)$, then $A$ can be obtained from any matrix in ${\mathcal A}^{\pm } (R,S)$ by a sequence of interchanges where all intermediary matrices are in  ${\mathcal A}^{\pm } (R,S)$.
 In particular, this is the case when  $m=n$ and $R=S=(1,1,\ldots,1)$, for then ${\mathcal A}^{\pm } (R,S)$ includes all permutation matrices. The next example illustrates Corollary \ref{cor:ansteex}.

   \begin{example}{\rm 
Consider $I_5$ and the following $(0,\pm 1)$-matrix with row and column sums equal to 1:
\[\left[\begin{array}{r|r|r|r|r}
1&&-1&1&\\ \hline
-1&1&1&& \\ \hline
-1&1&1&-1&1 \\ \hline   
1&&-1&1& \\ \hline   
1&-1&1&& \end{array}\right].\]
Then by interchanges we get
\[ \left[\begin{array}{r|r|r|r|r}
1&&-1&1&\\ \hline
-1&1&1&& \\ \hline
-1&1&1&-1&1 \\ \hline   
1&&-1&1& \\ \hline   
1&-1&1&& \end{array}\right]\rightarrow \left[\begin{array}{r|r|r|r|r}
1&&&&\\ \hline
-1&1&&1& \\ \hline
-1&1&1&-1&1 \\ \hline   
1&&-1&1& \\ \hline   
1&-1&1&& \end{array}\right]\rightarrow
\left[\begin{array}{r|r|r|r|r}
1&&&&\\ \hline
-1&1&&1& \\ \hline
-1&1&1&& \\ \hline   
1&&-1&&1 \\ \hline   
1&-1&1&& \end{array}\right]\rightarrow\]
\[\left[\begin{array}{r|r|r|r|r}
1&&&&\\ \hline
-1&1&&1& \\ \hline
&1&&& \\ \hline   
&&&&1 \\ \hline   
1&-1&1&& \end{array}\right]\rightarrow
\left[\begin{array}{r|r|r|r|r}
1&&&&\\ \hline
&&&1& \\ \hline
&1&&& \\ \hline   
&&&&1 \\ \hline   
&&1&& \end{array}\right].
\]
}\hfill{$\Box$}
\end{example}

\medskip
The next theorem characterizes when $\mc{A}(R,S)=\mc{A}^{\pm}(R,S)$, and shows (when the class is nonempty) the structure of a certain matrix in that class.

\begin{theorem}
 \label{thm:equal-ARS-pm} Let $R$ and $S$ be such that $\mc{A}(R,S)\ne\emptyset$. Then
 $\mc{A}(R,S)=\mc{A}^{\pm}(R,S)$ if and only if there are integers $k, l, m_i, n_i \ge 0$ $(i \le 3)$ with $p_i \ge 2$ $(i \le k+l)$ such that  $m=m_1+m_2+m_3$, $n=n_1+n_2+n_3$, 
\[
   k+\sum_{i=k+1}^{k+l} p_i = m_3, \;\; \mbox{\rm and} \;\; l+\sum_{i=1}^{k} p_i = n_3. 
\]     
and $R$ is a permutation of 
 \[
  (\underbrace{n,  \ldots, n}_{m_1}, \underbrace{n-1,  \ldots, n-1}_{m_2+m_3-k}, \underbrace{n-p_1,  \ldots, n-p_k}_{k})
\] 
and $S$ is a permutation of 
 \[
  (\underbrace{m, \ldots, m}_{n_1}, \underbrace{m-1,  \ldots, m-1}_{n_2+n_3-l}, \underbrace{m-p_{k+1},  \ldots, m-p_{k+l}}_{l}).
\] 
\end{theorem}
\begin{proof}
Let $A \in \mc{A}(R,S)$. Suppose there does not exist a matrix in
$\mc{A}^{\pm}(R,S)$ having a $-1$.  Then every 0 in $A$ must  either be (i) the only 0 in its row, (ii)
the only 0 in its column, or (iii) the only 0 in its row and the only 0 in its column. Otherwise
$A$ has a $2 \times 2$ submatrix with at most one 1 and then an interchange creates a matrix in $\mc{A}^{\pm}(R,S)$ with a $-1$.
So for each $0$ of $A$,  either it is the only 0 in its row or the only
0 in its column, or it is the only 0 in its row and the only 0 in its column. Thus the row and column sum vector of $A$ is of  the form given in the theorem. It is easy to see that if $R$ and $S$ are   of this form,  then  $\mc{A}(R,S)\ne \emptyset$ and  $\mc{A}^{\pm}(R,S)=\mc{A}(R,S)$.
\end{proof} 
 
\medskip
\begin{example}{\rm 
The following $6 \times 8$ matrix $A$ has row sum vector $R=(8,7,7,5,7,7)$ and column sum vector $S=(6,6,5,5,5,5,5,4)$ satisfying the properties specified in Theorem \ref{thm:equal-ARS-pm}.  A $6 \times 8$ $(0,\pm 1)$-matrix with these row and column sum vectors cannot contain a $-1$. 
\[
A=
\left[
\begin{array}{rr|rr|rrrr}
1&1&1&1&1&1&1&1\\ \hline 
1&1&0&1&1&1 &1&1\\ 
1&1&1&0&1&1&1&1 \\ \hline   
1&1&1&1&0 &0&0&1 \\   
1&1&1&1&1&1 &1&0 \\
1&1&1&1&1&1 &1&0
\end{array}
\right].
\]

}
\end{example} \endproof

To conclude this section, we consider another kind of question for the class  $\mc{A}^{\pm}(R,S)$, i.e., the 
$(0,\pm 1)$-matrices with row sum vector $R$ and column sum vector $S$. The following theorem determines the convex hull of this class (in the space $M_n$ of all $n\times n$ real matrices).

\begin{theorem}
 \label{thm:plus-minus-conv} 
  Then the convex hull of ${\mathcal A}^{\pm}(R,S)$ equals the set of $n \times n$ real matrices $A=[a_{ij}]$ satisfying
 \begin{equation}
 \label{eq:pm-conv}
  \begin{array}{cl} \vspace{0.05cm}
     \sum_{j=1}^n a_{ij} =r_i &(1\le i \le n), \\ \vspace{0.05cm}
       \sum_{i=1}^n a_{ij} =s_j &(1\le j \le n), \\ \vspace{0.05cm}
     -1 \le a_{ij} \le 1 &(1\le i,j \le n).
   \end{array}  
 \end{equation}
\end{theorem}
\begin{proof}
 This follows from the fact that the vertex-edge incidence matrix of a bipartite graph is totally unimodular, see \cite{Schrijver1986}  (Section 19.3). In fact, this general fact implies that each extreme point $A=[a_{ij}]$  of the polyhedron defined by (\ref{eq:pm-conv}) is integral, so $A$ is a $(0, \pm 1)$-matrix satisfying the equations in (\ref{eq:pm-conv}). Therefore the set of extreme points is equal to ${\mathcal A}^{\pm}(R,S)$.
\end{proof}

\section{Bruhat order}
\label{sec:Bruhat}

We return to sign-restricted matrices.
  Recall that ${\mathcal S}^+_{m,n}$ denotes the set of $m\times n$ $(0,1)$-SRMs, equivalently, the set of $m\times n$ SRMs without any $-1$'s. Thus the matrices in ${\mathcal S}^+_{m,n}$
have at most one 1 in each column and there is no restriction on the number of 1's in each row. The matrices in ${\mathcal S}^+_{m,n}$  are the incidence matrices of an {\it ordered partition} $(X_1,X_2,\ldots,X_m)$ of a subset of $\{1,2,\ldots,n\}$ in which, contrary to the usual definition of a partition, some of the parts $X_i$ may be empty. Two extreme cases are
$(\emptyset,\emptyset,\ldots,\emptyset)$ corresponding to the zero matrix $O_{m,n}$ in ${\mathcal S}^+_{m,n}$, and $(\{1,2,\ldots,n\}, \emptyset,\ldots,\emptyset)$ corresponding  to the matrix in ${\mathcal S}^+_{m,n}$  whose first row is all $1$'s and other rows are all $0$'s. 
We can also think of ${\mathcal S}^+_{m,n}$ as a generalization of the set ${\mathcal P}^*_{m,n}$ of $m\times n$ subpermutation matrices
(or, when $m=n$,  the set of $n\times n$ permutation matrices ${\mathcal P}_{n}$) where the restriction of at most one $1$ in each row is removed, but the restriction of at most one 1 in each column is retained. 

Let $n$ be a positive integer. Consider the partially ordered set (actually a distributive lattice) $(\mc{Q}_n,\subseteq)$ of subsets of $\{1,2,\ldots,n\}$ ordered by inclusion. We may identify the elements of $(\mc{Q}_n,\subseteq)$ with the partially ordered set of $2^n$\ $n$-tuples of $0$'s and $1$'s  where $(a_1,a_2,\ldots,a_n)\le (b_1,b_2,\ldots,b_n)$ if and only if $a_{i}\le b_{i}$ for $i=1,2,\ldots,n$.
As is well known, the $n\times n$ permutation matrices are in bijective correspondence with the saturated chains of $(\mc{Q}_n,\subseteq)$  from $\emptyset$ to $\{1,2,\ldots,n\}$; for instance, if $n=4$, then
\[(0,0,0,0)<(0,0,1,0)<(1,0,1,0)<(1,0,1,1)<(1,1,1,1),\]
equivalently,
\[\emptyset\subset \{3\}\subset \{1,3\}\subset \{1,3,4\}\subset \{1,2,3,4\}\]
and this 
corresponds to the permutation $(3,1,4,2)$ and the permutation matrix
\[\left[\begin{array}{c|c|c|c} &&1&\\ \hline 1&&&\\ \hline &&&1\\ \hline &1&&\end{array}\right].\]

There is a similar equivalence for the matrices in ${\mathcal S}^+_{m,n}$  which we now discuss. Consider as above the partially ordered set
$(\mc{Q}_n,\subseteq)$. A {\it  multichain of length $m$} in $(\mc{Q}_n,\subseteq)$ is
a sequence of subsets of $\{1,2,\ldots,n\}$ of the form
\[\emptyset=X_0\subseteq X_1\subseteq X_2\subseteq\cdots\subseteq 
X_m.\]
Notice that the definition of a multichain implies that it starts with $\emptyset$. Since column sums may equal 0, a multichain need not end with $X_m=\{1,2,\ldots,n\}$.
Also in contrast to the usual notion of a chain in a partially ordered set,
in a multichain there may be repeats in the chain.\footnote{If one thinks of a chain of length $m$ as a path of length $m$ of edges in a graph, a multichain may have loops, perhaps more than 1,  at any vertex of that path.}
Let $\mc{C}_{m,n}$ be the set of  all multichains of $(\mc{Q}_n,\subseteq)$ of length $m$. In terms of the identification of $\mc{Q}_n$ as $n$-tuples of $0$' s and $1$'s, a multichain allows for the possibility that successive $n$-tuples are equal.

Generalizing the above, the  matrices in  ${\mathcal S}^+_{m,n}$ are in bijective correspondence with the multichains in $\mc{C}_{m,n}$; for instance,  if $m=4$ and $n=6$, then
\[(0,0,0,0,0,0)\le (0,1,0,0,1,0)\le (0,1,1,0,1,0)\le (0,1,1,0,1,0)\le (1,1,1,1,1,1),\]
equivalently,
\[\emptyset\subseteq \{2,5\}\subseteq  \{2,3,5\}\subseteq  \{2,3,5\}\subseteq \{1,2,3,4,5,6\}\] and this
corresponds to the matrix in ${\mathcal S}^+_{4,6}$ given by
\[\left[\begin{array}{c|c|c|c|c|c}
&1&&&1&\\ \hline
&&1&&&\\ \hline
&&&&&\\ \hline
1&&&1&&1\end{array}\right].\]

There is a partial order, denoted by $\le_B$ and called the {\it Bruhat-order}, on the set $\mc{P}_n$ of $n\times n$ permutation matrices (and other classes of matrices as well including ASMs) which can be defined as follows:

For an $m\times n$ matrix $A=[a_{ij}]$, let the {\it sum-matrix}\footnote{Also called the {\em corner sum matrix} in the literature on ASMs.} of $A$ be
$\Sigma(A)=[\sigma_{ij}(A)]$ where
\[\sigma_{ij}(A)=\sum_{1\le p\le i,\,1\le q\le j}  a_{pq},\quad (1\le i\le m,1\le j\le n), \] the sum of the entries of $A$ in its leading $i\times j$ submatrix.
Then 
\[A_1\le_B A_2 \mbox{ provided that }\Sigma(A_1)\ge \Sigma(A_2)\mbox{ (entrywise).}\]
The partially ordered set $(\mc{P}_n,\le_B)$ is not a lattice  if $n\ge 3$. The {\it Dedekind-MacNeille completion} of $(\mc{P}_n,\le_B)$, the (unique up to isomorphism) smallest lattice extension of $(\mc{P}_n,\le_B)$, was shown  by Lascoux and Sch\H utzenberger \cite{LS} to be  the Bruhat order on the set $\mc{A}_n$ of $n\times n$ ASMs:
\[A_1\le_B A_2\mbox{ if and only if }\Sigma(A_1)\ge \Sigma(A_2),\quad
(A_1,A_2\in \mc{A}_n).\] The minimum element of  the lattice $(\mc{A}_n,\le_B)$ is the $n\times n$ identity matrix and the maximum element is the $n\times n$ anti-identity matrix $L_n$.  In \cite{F} the  Dedekind-MacNeille completion of the poset of partial injective functions was determined. This is similar, but not identical, to our
result in Theorem \ref{thm:D-M-completion}, since the posets considered in \cite{F} are
subposets of ours.

We can extend the Bruhat order, as defined above using the sum-matrix, to  ${\mathcal S}^+_{m,n}$ and  ${\mathcal S}_{m,n}$, thereby obtaining two partially ordered sets
$ ({\mathcal S}^+_{m,n},\le_B)$ and $ ({\mathcal S}_{m,n},\le_B)$.
%{\bf What is a characterization of the sum-matrices of ${\mathcal S}^+_{m,n}$ and  ${\mathcal S}_{m,n}$? }

\begin{example}{\rm \label{ex:bruhat1}
Let $m=n=2$. The set of matrices in $ ({\mathcal S}^+_{2,2},\le_B)$ along with their sum-matrices, indicated by $\rightarrow$, is:
\[(a)\
\left[\begin{array}{cc} 0&0 \\ 0&0\end{array}\right]\rightarrow \left[\begin{array}{cc} 0&0 \\ 0&0\end{array}\right],\quad (b)\
\left[\begin{array}{cc} 1& 0\\ 0&0\end{array}\right]\rightarrow \left[\begin{array}{cc} 1&1 \\ 1&1\end{array}\right],\]
\[(c)\
\left[\begin{array}{cc} 0&1 \\ 0&0\end{array}\right]\rightarrow \left[\begin{array}{cc} 0&1 \\ 0&1\end{array}\right],\quad (d)\
\left[\begin{array}{cc} 0&0 \\ 1&0\end{array}\right]\rightarrow \left[\begin{array}{cc} 0&0 \\ 1&1\end{array}\right],\]
\[ (e)\
\left[\begin{array}{cc} 0&0 \\ 0&1\end{array}\right]\rightarrow \left[\begin{array}{cc} 0&0 \\ 0&1\end{array}\right],\quad  (f)\
\left[\begin{array}{cc} 1&0 \\ 0&1\end{array}\right]\rightarrow \left[\begin{array}{cc} 1&1 \\ 1&2\end{array}\right],\]
\[ (g)\
\left[\begin{array}{cc} 0&1 \\ 1&0\end{array}\right]\rightarrow \left[\begin{array}{cc} 0&1 \\ 1&2\end{array}\right],\quad  (h)\
\left[\begin{array}{cc} 1&1 \\ 0&0\end{array}\right]\rightarrow \left[\begin{array}{cc} 1&2 \\ 1&2\end{array}\right],\]
\[ (i)\
\left[\begin{array}{cc} 0&0 \\ 1&1\end{array}\right]\rightarrow \left[\begin{array}{cc} 0&0 \\ 1&2\end{array}\right].\]

 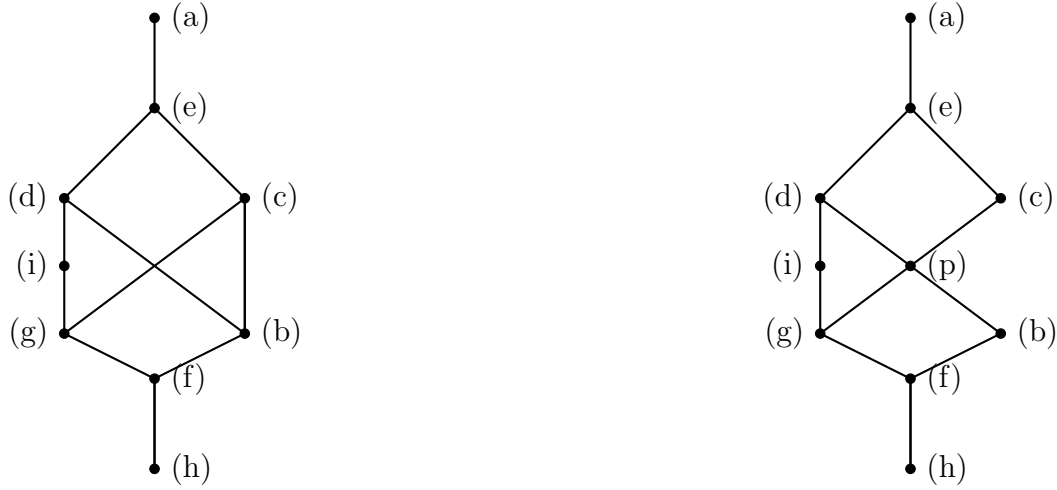
\begin{figure}
%\begin{center}
\centering
\begin{minipage}[t]{0.33\textwidth}
\begin{tikzpicture}[scale=0.60]
\filldraw (0,0) circle (3pt) node[right=2pt] {(h)};
\filldraw (0,2) circle (3pt) node[right=2pt] {(f)};
\filldraw (2,3) circle (3pt) node[right=2pt] {(b)};
\filldraw (2,6) circle (3pt) node[right=2pt] {(c)};
\filldraw (0,8) circle (3pt) node[right=2pt] {(e)};
\filldraw (0,10) circle (3pt) node[right=2pt] {(a)};
\filldraw (-2,3)circle (3pt) node[left=2pt] {(g)};
\filldraw (-2,4.5) circle (3pt) node[left=2pt] {(i)};
\filldraw (-2,6) circle (3pt) node[left=2pt] {(d)};

 \draw[thick] (0,0) -- (0,2);
 \draw[thick] (2,3) -- (0,2);
 \draw[thick] (2,3) -- (2,6);
 \draw[thick] (0,8) -- (-2,6);
 \draw[thick] (0,8) -- (0,10);
 \draw[thick] (-2,3) -- (-2,4.5);
 \draw[thick] (-2,4.5) -- (-2,6);
 \draw[thick] (-2,3) -- (0,2);
 \draw[thick] (0,0) -- (0,2);
 \draw[thick] (0,8) -- (2,6);
  \draw[thick] (-2,3) -- (2,6);
  \draw[thick] (2,3) -- (-2,6);
   \draw[thick] (2,3) -- (2,6);

  \end{tikzpicture}
\end{minipage}\hfill
\begin{minipage}[t]{0.33\textwidth}
  \begin{tikzpicture}[scale=0.60]
\filldraw (0,0) circle (3pt) node[right=2pt] {(h)};
\filldraw (0,2) circle (3pt) node[right=2pt] {(f)};
\filldraw (2,3) circle (3pt) node[right=2pt] {(b)};
\filldraw (2,6) circle (3pt) node[right=2pt] {(c)};
\filldraw (0,8) circle (3pt) node[right=2pt] {(e)};
\filldraw (0,10) circle (3pt) node[right=2pt] {(a)};
\filldraw (-2,3)circle (3pt) node[left=2pt] {(g)};
\filldraw (-2,4.5) circle (3pt) node[left=2pt] {(i)};
\filldraw (-2,6) circle (3pt) node[left=2pt] {(d)};
\filldraw (0,4.5) circle (3pt) node[right=2pt] {(p)};

 \draw[thick] (0,0) -- (0,2);
 \draw[thick] (2,3) -- (0,2);
 %\draw[thick] (2,3) -- (2,6);
 \draw[thick] (0,8) -- (-2,6);
 \draw[thick] (0,8) -- (0,10);
 \draw[thick] (-2,3) -- (-2,4.5);
 \draw[thick] (-2,4.5) -- (-2,6);
 \draw[thick] (-2,3) -- (0,2);
 \draw[thick] (0,0) -- (0,2);
 \draw[thick] (0,8) -- (2,6);
  \draw[thick] (-2,3) -- (2,6);
   \draw[thick] (2,3) -- (-2,6);
  \end{tikzpicture}
\end{minipage}\hfill
% \end{center}
 \caption{Hasse diagram of $(\mc{S}^+_{2,2},\preceq_B)$ and $(\mc{S}_{2,2},\preceq_B)$.}
 \label{fig:one}
\end{figure}

Examining Figure \ref{fig:one} we see that every pair of elements except $\{b,g\}$ has a unique LUB (least upper bound in the Bruhat order), and every pair of elements except $\{c,d\}$ has a unique GLB (greatest lower bound in the Bruhat order).
There is only one matrix in ${\mathcal S}_{2,2}$ that is not in ${\mathcal S}^+_{2,2}$, namely the matrix $(p)$ indicated below with its sum-matrix:
\[ (p)\
\left[\begin{array}{cr} 0&1 \\ 1&-1\end{array}\right]\rightarrow \left[\begin{array}{cc} 0&1 \\ 1&1\end{array}\right].\] 
We see that $\Sigma(p)\ge \Sigma(c)$, $\Sigma(p)\ge\Sigma(d)$, $\Sigma(p)\le \Sigma(b)$, and $\Sigma(p)\le \Sigma(g)$. We conclude that $(p)$ is the GLB of $(c)$ and $(d)$, and $(p)$ is the LUB of $(b)$ and $(g)$ in $(\mc{S}_{2,2}),\le_B)$  and that $(\mc{S}_{2,2},\le_B)$ is a lattice; indeed  $(\mc{S}_{2,2},\le_B)$ is therefore the Dedekind-MacNeille completion of $(\mc{S}^+_{2,2},\le_B)$; see Figure \ref{fig:one}.
}\hfill{$\Box$}\end{example}

It follows that $(\mc{P}_n,\le_B)$ is a subposet of
$(\mc{S}_{n,n},\le_B)$,  and $(\mc{P}^*_{m,n},\le_B)$ is  a subposet of
$(\mc{S}_{m,n},\le_B)$. Clearly, the maximal element of both  $(\mc{P}_n,\le_B)$ and  $(\mc{P}^*_{m,n},\le_B)$ is $O_{m,n}$, and the minimal element is the $m\times n$ matrix
$\Upsilon_{m,n}$ with all 1's in row 1 and 0's elsewhere. We have that
$\Sigma(\Upsilon_{m,n})$ has all of its rows equal to $(1,2,\ldots,n)$ and hence the sum of the entries of $\Sigma(\Upsilon_{m,n})$  equals 
$m{{n+1}\choose 2}$.

We use the notations $A\vee B ={\rm LUB}\{A,B\}$ and $A\wedge B={\rm GLB}\{A,B\}$ for $A$ and $B$ in a lattice.
Also $a \vee b = \max\{a,b\}$ and $a \wedge b = \min\{a,b\}$ for real numbers $a, b$.

\begin{theorem}
\label{thm:D-M-completion}
  $(\mc{S}_{m,n}, \le_B)$ is a distributive lattice, and it  is  the Dedekind-MacNeille completion of $(\mc{S}^+_{m,n},\le_B)$.
\end{theorem}
\begin{proof}
  Let $A, A' \in \mc{S}_{m,n}$ and let $S=\Sigma(A)=[s_{ij}]$, $S'=\Sigma(A')=[s'_{ij}]$.  Since $f: M_{m,n} \rightarrow M_{m,n}$ given by $C \rightarrow \Sigma(C)$ is an isomorphism, there is a unique matrix $C=[c_{ij}]$ such that $\Sigma(C)=T$ where $T=\max\{S,S'\}$. Then $C$ is integral; this follows from the facts that $T$ is integral and $f^{-1}$ maps integral matrices to integral matrices. We show   further properties of $C$.

Let $1 <i \le m$, $1<j \le n$. Consider the leading $i \times j$ submatrix of $A$ and partition it by the first $(j-1)$ columns and the last one, and similar for the rows. The sum of the entries in each of the four blocks of this submatrix are indicated in the following diagram  %
 \begin{center}
 \begin{tabular}{|c|c|} \hline
      $s_{i-1,j-1}$ & $s_1$  \\  \hline
     $s_2$ & $a_{ij}$ \\  \hline
 \end{tabular}   
 \end{center}
 where $s_1=\sum_{r=1}^{i-1} a_{rj}$ and $s_2=\sum_{s=1}^{j-1} a_{is}$.  
 Similarly, for $A'$, we obtain the four sums $s'_{i-1,j-1}$, $a'_{ij}$, $s'_1=\sum_{r=1}^{i-1} a'_{rj}$ and $s'_2=\sum_{s=1}^{j-1} a'_{is}$. Here $s_1, s'_1 \in \{0,1\}$  as  $A, A' \in \mc{S}_{m,n}$.  
Similarly, $s_1+a_{ij}, s'_1+a'_{ij}  \in \{0,1\}$ and each of the four numbers $s_2, s_2+a_{ij}, s'_2, s'_2+a'_{ij}$ are nonnegative. 

 Consider the matrix $C=[c_{ij}]$ defined above.
  Then
\begin{equation}
\label{eq:c_ij}
  c_{ij} = t_{ij}-t_{i-1,j}-t_{i,j-1}+t_{i-1,j-1}
\end{equation}
where 
 \[
\begin{array}{ll}
   t_{i-1,j-1} &=  s_{i-1,j-1}  \vee s'_{i-1,j-1} \\*[\smallskipamount]
   t_{i-1,j} &=  (s_{i-1,j-1}+s_1)  \vee (s'_{i-1,j-1}+s'_1) \\*[\smallskipamount]
   t_{i,j-1} &=  (s_{i-1,j-1}+s_2)  \vee (s'_{i-1,j-1}+s'_2) \\*[\smallskipamount]
   t_{ij} &=  (s_{i-1,j-1}+s_1+s_2+a_{ij})  \vee (s'_{i-1,j-1}+s'_1+s'_2+a'_{ij}).
\end{array}
\] 
Observe that $t_{i-1,j} -t_{i-1,j-1}  \in \{0,1\}$ as $s_1, s'_1 \in \{0,1\}$. Similarly, 
$t_{ij} -t_{i,j-1}  \in \{0,1\}$ as $s_1+a_{ij}, s'_1+a'_{ij} \in \{0,1\}$. Therefore 
\[
     c_{ij} = (t_{ij}-t_{i,j-1})-(t_{i-1,j} -t_{i-1,j-1}) \in \{-1,0,1\}.
\]
Thus, $C$ is a $(0,\pm 1)$-matrix. This also shows that 
\begin{equation}
\label{eq:C-properties}
   c_{ij} = \left\{
     \begin{array}{rl}
             1 &\mbox{\rm  when $t_{ij}=t_{i,j-1}+1$ and $t_{i-1,j}=t_{i-1,j-1}$}, \\
           -1 &\mbox{\rm  when $t_{ij}=t_{i,j-1}$ and $t_{i-1,j}=t_{i-1,j-1}+1$}, \\
             0 &\mbox{\rm otherwise. }
     \end{array}   
     \right.      
\end{equation}
Here $t_{0j}=0$ for each $j$ and $t_{i0}=0$ for each $i$ (as for $A$ and $A'$). 

We now prove that $C \in \mc{S}_{m,n}$. First, let $1<j \le n$. Define $I^+_j=\{i: t_{ij} =t_{i,j-1}+1\}$, and note that $0 \not \in I^+_j$. 
 Assume $i-1 \not \in I^+_j, i \in I^+_j,  i+1 \in I^+_j, \ldots,  k \in I^+_j, k+1 \not \in I^+_j$. From (\ref{eq:C-properties}) we get
 \[
   c_{ij}=1, \; c_{i+1,j}=c_{i+2,j}= \cdots = c_{kj}=0, \; c_{k+1,j}=-1.
 \]
 This implies that the nonzeros in column $j$   alternates between 1 and $-1$, starting with a $1$ (if any) as $0 \not \in I^+_j$. Also, the first column of $T$ is the maximum of the first column in $A$ and the first column in $A'$, and therefore the first column in $C$ is either zero or it contains a single 1.

Next, let $1\le i \le m$ and $1 \le k \le n$. Then 
\[
\begin{array}{rl}
  \sum_{j=1}^k c_{ij} & =t_{ik}-t_{i-1,k} \\*[\smallskipamount]
     &=(s_{ik} \vee s'_{ik})-(s_{i-1,k} \vee s'_{i-1,k}) \\*[\smallskipamount]
     & \ge 0
 \end{array} 
\]
as $s_{ik} \ge s_{i-1,k}$ and $s'_{ik} \ge s'_{i-1,k}$, again due to $A, A' \in \mc{S}_{m,n}$. This proves that $C \in \mc{S}_{m,n}$.

Thus, we have that each pair $A, A'$ of matrices in $\mc{S}_{m,n}$ has a unique greatest lower bound (meet) in the Bruhat order, given by the matrix $C$ above. Since $(\mc{S}_{m,n}, \le_B)$ is a finite partially ordered set (or use a similar argument)  the corresponding statement for least upper bound (join) holds as well. Therefore $(\mc{S}_{m,n}, \le_B)$ is a lattice. In order that $(\mc{S}_{m,n}, \le_B)$ be distributive, we must have
\[A_1\wedge (A_2\vee A_3)=(A_1\wedge A_2)\vee (A_1\wedge A_3)\quad (A_1,A_2,A_3\in \mc{S}_{m,n}).\]
The corresponding property for the real numbers with the usual $\le$ order relation holds. Since the join and meet in $(\mc{S}_{m,n}, \le_B)$ is componentwise on the sum-matrices, it follows that $(\mc{S}_{m,n}, \le_B)$ is distributive.

  It remains to prove  that $(\mc{S}_{m,n}, \le_B)$ is the Dedekind-MacNeille completion of $(\mc{S}^+_{m,n},\le_B)$. This will follow by showing that any given matrix $A \in \mc{S}_{m,n}$ is the meet of some set of  matrices in $\mc{S}^+_{m,n}$.
  
  Let $S=\Sigma(A)=[s_{ij}]$. For $i=1, 2, \ldots, m$, let $S^{(i)}$ be the $m \times n$ matrix whose first $(i-1)$ rows are zero and each of the $(m-i+1)$ remaining rows are equal to row $i$ of $S$ (so these rows are equal). Then clearly 
\[
   S= \max\{S^{(1)}, S^{(2)}, \ldots, S^{(m)}\}
\]   
as $S$ has monotone columns. Now, $S^{(i)}=\Sigma(A^{(i)})$ where $A^{(i)}$ is the $(0,1)$-matrix whose only nonzero row is row $i$, and it contains a 1 in position $(i,j)$ precisely when row $i$ in $S$ has an increase in column $j$, i.e., $s_{i,j-1}<s_{ij}$  $(j \le n)$ (where we think of a zero'th row and column of $S$ to contain only zeros). Note that every increase in $S$ is 1 as the columns in $A$ have alternating signs, so $A^{(i)} \in \mc{S}^+_{m,n}$. Therefore $A$ is the meet of $\{A^{(i)}: i \le m\}$ in the Bruhat order.  
Thus  $({\mathcal S}_{m,n},\le_B)$ is the Dedekind-MacNeille completion of $({\mathcal S}_{m,n}^+,\le_B)$.
\end{proof}

The construction in the final part of the proof is illustrated by the next example. 

\begin{example} \label{ex:feng}
{\rm 
Consider the following matrix $A$ in  ${\mathcal S}_{6,6}$ with its $\Sigma(A)$:

\[A=\left[\begin{array}{r|r|r|r|r|r}
&1&&1&1&\\ \hline
&&1&-1&&1\\ \hline
1&&-1&1&&-1\\ \hline
&&1&&-1&1\\ \hline
&&&&1&-1\\ \hline
&&&&&1\end{array}\right]\rightarrow \Sigma(A)=
\left[\begin{array}{c|c|c|c|c|c}
0&1&1&2&3&3\\ \hline
0&1&2&2&3&4\\ \hline

1&2&2&3&4&4\\ \hline

1&2&3&4&4&5\\ \hline

1&2&3&4&5&5\\ \hline

1&2&3&4&5&6\end{array}\right].
\]
Then  $A$ is the meet of  six matrices $A^{(1)}, A^{(2)}, \ldots,  A^{(6)}$ in ${\mathcal S}_{6,6}^+$:
 These matrices have only one nonzero row, with a different row number in each case, as specified below:
  
 row 1: $[0, 1, 0, 1, 1, 0]$
 
 row 2: $[0, 1, 1, 0, 1, 1 ]$
 
 row 3: $[1, 1, 0, 1, 1, 0]$
 
 row 4: $[1, 1, 1, 1, 0, 1]$
 
 row 5: $[1, 1, 1, 1, 1, 0]$
 
 row 6: $[1, 1, 1, 1, 1, 1]$.
 
 These rows are obtained from the increases in the corresponding rows of $\Sigma(A)$ where one has to imagine a zeroth column of all 0's.
 So for instance, 
 \[
 \Sigma(A^{(3)})=
 \left[\begin{array}{c|c|c|c|c|c}
0&0&0&0&0&0\\ \hline
0&0&0&0&0&0 \\ \hline
1&2&2&3&4&4\\ \hline
1&2&2&3&4&4\\ \hline
1&2&2&3&4&4\\ \hline
1&2&2&3&4&4\end{array}\right].\]
 Then 
\[\max\{\Sigma(A^{(1)}),\Sigma(A^{(2)}),\ldots, \Sigma(A^{(6)})\}=\Sigma(A)\]
and in $({\mathcal S},\le_B)$, the  matrix $A$ is the meet of $A^{(1)}, A^{(2)}, \ldots,  A^{(6)}$.
} \endproof
\end{example}

A finite lattice has a unique smallest element called its {\it zero element}.
A {\it join-irreducible element} of a finite lattice is a nonzero element of the lattice which cannot be expressed as the join of two elements different from it.  A nonzero element is join-irreducible if and only if it covers exactly one element in the lattice. A {\it meet-irreducible element} is defined analogously. These and other properties can be found in \cite{DP90}.
It is straightforward to verify that the set of  meet-irreducible elements of ${\mathcal S}_{m,n}$ are the matrices in ${\mathcal S}_{m,n}^+$ with exactly one nonzero row. This follows as in Example \ref{ex:feng}, and using the observation that two distinct $(0,1)$-vectors of the same size have different sum-matrices. 

By Birkhoff's representation theorem for finite distributive lattices (Theorem 8.17 in \cite{DP90}), the  lattice  $({\mathcal S}_{m,n},\le_B)$ can be respresented as the lattice $({\mathcal J}_{m,n},\subseteq)$ whose elements are the set
${\mathcal J}_{m,n}$ of join-irreducible elements of 
$({\mathcal S}_{m,n},\le_B)$ where we identify each element $x$ with the set $\{u\le_B x: u\in {\mathcal J}_{m,n}\}$ of join-irreducibles below $x$, and the partial order is that of set-containment. Equivalently, we   represent $({\mathcal S}_{m,n},\le_B)$  as the lattice  whose elements are the set
${\mathcal M}_{m,n}$ of meet-irreducible elements of 
$({\mathcal S}_{m,n},\le_B)$ where we identify each element $x$ with the set $\{u\ge_B x: u\in {\mathcal M}_{m,n}\}$ of meet-irreducibles above $x$, and the partial order is that of reverse set-containment.

\begin{example}\label{ex:bruhat1more}{\rm We continue with the Example \ref{ex:bruhat1} and show the two representations described above, using $(0,1)$-vectors instead of sets, where we omit the $0$'s for clarity:
\[({\mathcal J}_{m,n},\subseteq):\quad
\begin{array}{c||c|c|c|c|c|c}
&(a)&(b)&(c)&(f)&(g)&(i)\\ \hline\hline
(a)&1&1&1&1&1&1\\ \hline
(b)&&1&&1&&\\ \hline
(c)&&1&1&1&1&\\ \hline
(d)&&1&&1&1&1\\ \hline
(e)&&1&1&1&1&1\\ \hline
(f)&&&&1&&\\ \hline
(g)&&&&1&1&\\ \hline
(h)&&&&&&\\ \hline
(i)&&&&1&1&1\\ \hline
(p)&&1&1&1&&\end{array}\]

\[({\mathcal M}_{m,n},\supseteq):\quad
\begin{array}{c||c|c|c|c|c|c}
&(b)&(c)&(d)&(e)&(h)&(i)\\ \hline\hline
(a)&&&&&&\\ \hline
(b)&1&1&1&1&&\\ \hline
(c)&&1&&1&&\\ \hline
(d)&&&1&1&&\\ \hline
(e)&1&&&&&\\ \hline
(f)&1&1&1&1&&1\\ \hline
(g)&&1&1&1&&1\\ \hline
(h)&1&1&1&1&1&1\\ \hline
(i)&&&1&1&&1\\ \hline
(p)&&1&1&1&&\end{array}\]

}\hfill{$\Box$}
\end{example}

Next, we characterize the sum-matrices for the class $\mc{S}^+_{m,n}$.

\begin{lemma}
 \label{lem:Bruhat-sum-char} 
  Let $S=[s_{ij}]$ be a nonnegative integral $m \times n$ matrix. Then $S$  is the sum-matrix of a matrix in $\mc{S}^+_{m,n}$ if and only if 
  \begin{equation}
   \label{eq:sum-char}
   \begin{array}{rll}
    s_{i,j-1} + s_{i-1,j} - s_{i-1,j-1}&\le s_{ij}      &(i \le m, j \le n), \\*[\smallskipamount]
    s_{mj} &\le s_{m,j-1}+1 &(j \le n), 
   \end{array} 
  \end{equation} 
  where we define $s_{0j}=s_{i0}=0$ for each $i$ and $j$.
\end{lemma}
\begin{proof}
Let $A \in \mc{S}^+_{m,n}$ and let  $S=\Sigma(A)$,  $S=[s_{ij}]$. Then the first set of constraints in (\ref{eq:sum-char}) holds as $A$ is nonnegative and $a_{ij}=s_{ij}-s_{i-1,j}-s_{i,j-1}+s_{i-1,j-1}$ for each $i,j$. 
The other constraints hold as  each column of $A$ has at most one  1.

Conversely, assume $S$ satisfies (\ref{eq:sum-char}). 
As mentioned, the linear map $T: M_{m,n} \rightarrow M_{m,n}$ given by $T(A)=\Sigma(A)$ is an isomorphism, there is a unique matrix $A \in M_{m,n}$ such that $T(A)=S$, and this $A=[a_{ij}]$ is given by $a_{ij}=s_{ij}-s_{i-1,j}-s_{i,j-1}+s_{i-1,j-1}$ for each $i,j$. This matrix $A$ is integral. Also, the first set of constraints in (\ref{eq:sum-char}) implies $a_{ij} \ge 0$ for each $i,j$. Moreover, the second set of constraints in (\ref{eq:sum-char}) gives  
\[
   \sum_{i=1}^m a_{ij} = s_{mj}-s_{m,j-1} \le 1
\]
which (as $A$ is nonnegative and integral) means that $A$ is a $(0,1)$-matrix with at most one 1 in every column, so $A \in \mc{S}^+_{m,n}$, as desired.
\end{proof}

Let $A\in {\mathcal A}(R,S)$. A {\em Bruhat interchange} (applied to $A$) is to replace a submatrix 
\[
 \left[
 \begin{array}{rr}
   0 & 1 \\
   1 & 0
 \end{array}  
 \right]
\]
by the identity matrix of order 2.

\begin{lemma}
\label{lem:Bruhat-equal}
Let $R$ be an $m$-vector of nonnegative integers with sum $n$, and let $S=(1,1,\ldots,1) \in \mb{R}^n$. 
Let $A,C\in {\mathcal A}(R,S)$. Then $\Sigma(A)\ge \Sigma(C)$ if and only if $C$ can be transformed into $A$ by Bruhat interchanges.
\end{lemma}
\begin{proof} 
  The assumptions on  $R$ and $S$ assure that $\mc{A}(R,S)$ is nonempty. 
	If $A$ can be transformed into $C$ by Bruhat interchanges, then clearly $\Sigma(A)\ge \Sigma(C)$. Now suppose that $\Sigma(A)\ge \Sigma(C)$. Let the $k$'th row be the first row where $ \Sigma(A)$ and $\Sigma(C)$ differ, and let $l$ be the first position in row $k$ where they differ. Thus $a_{kl}=1$ and $c_{kl}=0$. Since $A$ and $C$ have the same row sum, let $t>l$ be the first position where $a_{kt}=0$ and $c_{kt}=1$. Consider the submatrices of $A$ and $C$ in the region determined by rows $k+1,\ldots,m$ and columns $l+1,\ldots,t$. Note that column $t$ of $A$ has a 1 in this region, since $A$ and $C$ agree in column $t$ above row $k$ and so both have only 0's in columns $t$ above row $k$. Thus $A$ has a 1 in this region. Consider the uppermost 1, say it is in position $(k',l')$.  Let $A'=[a'_{ij}]$ be obtained from $A$ by interchanging columns $l$ and $l'$, so, as each column has exactly one 1, this is the  inverse of a Bruhat interchange. Therefore $A \le_B A'$ and $A' \not = A$. We prove that $\Sigma(A') \ge \Sigma(C)$. Note that 
\[
   \Sigma(A')_{ij}= \left\{ \begin{array}{ll}
                                        \Sigma(A)_{ij}-1 & \mbox{\rm for $k \le i <k', \,l \le j <l'$} \\
                                        \Sigma(A)_{ij} & \mbox{\rm otherwise}.
                                     \end{array}   
                             \right.           
\]
So, if $k \le i <k', \,l \le j <l'$, then 
$\Sigma(A)_{ij}= \Sigma(A)_{il'}=0$ and  
\[
    \Sigma(A')_{ij}=\Sigma(A)_{ij}-1= \Sigma(A)_{il'}-1\ge \Sigma(C)_{il'}-1\ge \Sigma(C)_{ij}.
\]
So $\Sigma(A') \ge \Sigma(C)$, and $A \le_B A' \le_B C$, as desired. Also, $A'$ and $C$ agree in one more position in row $k$, namely, $(k,l)$. The desired result now follows by  induction.
\end{proof}

\begin{theorem}
	\label{thm:Bruhat-S} 
	Consider the partially ordered set $(\mc{S}^+_{m,n},\le_B)$, and let $A, C \in \mc{S}^+_{m,n}$. Then 
	$A \le_B C$ if and only if $A$ can be obtained from $C$ by a sequence of operations of the form
	
	$(i)$ a Bruhat interchange, or 
	
	$(ii)$ replacing a zero column by a coordinate vector, or 
	
	$(iii)$ replacing a column equal to  $e_k$ by  $e_i$ where $i<k$, or
	
	$(iv)$ interchanging a nonzero column $j$ with a zero column $k$, where $k<j$.
	
\end{theorem}

\begin{proof}
We first extend $A$ and $A'$ and $C$ to $C'$  by appending a new row of $0$'s and $1$'s so that all column sums of $A'$ and $C'$ are now equal to 1. Let the row sums of $A'$ be
$p_1, p_2, \ldots, p_{m+1}$ and the row sums of $C'$ be $q_1, q_2, \ldots, q_{m+1}$. Since $A'$ and $C'$ have exactly one 1 in each column, we have
\[
	p_1+p_2+\cdots +p_{m+1}=q_1+q_2+\cdots +q_{m+1}\]
\[\	(p_1-q_1)+(p_2-q_2)+\cdots+(p_{m+1}-q_{m+1})=0\]
\[	\sum_{i=1}^{m+1}(p_i-q_i)^+=\sum_{i=1}^{m+1}(q_i-p_i)^+.\]
Let this common value in the last equation be $t$. We extend $A'$ and $C'$ to $(m+1)\times (n+t)$ $(0,1)$-matrices by including columns with exactly one 1 so that the resulting matrices $A''$ and $C''$
have the same row sum vector $R$.  In each case we use the earliest column as we go down the rows. Thus $A''$ and $C''$ belong to the class ${\mathcal A}(R,S)$ where $S$ is a vector of all 1's. Moreover, as $\Sigma (A)\ge \Sigma(C)$, it follows that  $\Sigma(A'')\ge \Sigma(C'')$. In fact, 
$\Sigma (A)_{m+1,j}= \Sigma(C)_{m+1,j}=j$ for each $j$. Also, by the construction, 
\[
   \Sigma (A'')_{ij} \ge \Sigma(C'')_{ij} \;\;(i\le m+1, \,n\le j \le n+t).
\]
Thus  by Lemma \ref{lem:Bruhat-equal}, $A''$ can be transformed to $C''$ by a sequence of inverse Bruhat interchanges.
There are four types of interchanges depending on where the corresponding $2\times 2$ matrix lies in
\[\left[\begin{array}{c|c}
	X_1&X_2\\ \hline
	X_3&X_4\end{array}\right].\]
	Here $X_3$ and $X_4$ have only one row. The relation between the position of the interchange and the type of operation in the theorem is now as follows: 
	\begin{itemize}
		\item Wholly in $X_1$, and so (i).
		\item In $X_1,X_2,X_3,X_4$, and so (ii).
			\item In $X_1$ and $X_2$, and so (iii).
		\item In $X_1$ and $X_3$ and so (iv).
	\end{itemize}
Hence if $\Sigma(A)\ge \Sigma(C)$, we can get from $A$ to $C$ by a sequence of inverse of the operations (i), (ii), (iii), and (iv).
	\end{proof}

We remark that the proof of Theorem \ref{thm:Bruhat-S}  actually contains an efficient algorithm which, for given matrices $A, C \in \mc{S}^+_{m,n}$ with   $A \le_B C$, constructs matrices $K^{(i)}$ with 
\[
    A = K^{(p)} \le _B  K^{(p-1)} \le _B \cdots \le_B K^{(1)} =C
\]
such that $K^{(i+1)}$ is constructed from $K^{(i)}$ ($1\le i <p$) by one of the four operations given in the theorem.

\medskip

\begin{corollary}
 \label{conj:Bruhat-S}
  Consider the partially ordered set $(\mc{S}^+_{m,n},\le_B)$, and let $A, C \in \mc{S}^+_{m,n}$. Then $A$ covers $C$ in $(\mc{S}^+_{m,n},\le_B)$
 if and only if $C$ can be obtained from $A$ by one of operations of the form
\begin{itemize}
 \item[\rm (i)] a Bruhat interchange $\left[\begin{array}{cc}0&1\\1&0\end{array}\right]\rightarrow
  \left[\begin{array}{cc}1&0\\0&1\end{array}\right]$ within consecutive rows and consecutive columns.
  \item[\rm (ii)] 
   replacing a zero column $n$ with a column with exactly one $1$ where this $1$ is in the last position  or, more generally, replacing column $n$ which has a $1$ in row $j$ with a column which has a $1$ in row $(j-1)$,
  \item[\rm (iii)] interchanging column $j$ with columns $(j-1)$ where the $1$ in column $j$ is in the last position and column $(j-1)$ is a zero column.
  \end{itemize}
\end{corollary}
\begin{proof} This  follows from Theorem \ref{thm:Bruhat-S} as these are the operations in that theorem which increase the sum of the entries of $\Sigma(C)$ by exactly 1.
\end{proof}

 \section{Polytope and decomposition}
 \label{sec:polytope}

 In \cite{SS} the {\em sign matrix polytope} ${\mathcal P}_{m,n}$ is defined as the convex hull of the matrices in ${\mathcal S}_{m,n}$ (the SRMs of size $m \times n$). It is stated in \cite{SS} that ``all $mn$ entries contribute to the dimension'' and  thus that the dimension of ${\mathcal P}_{m,n}$ is $mn$ for $m>1$. In fact,  every $m\times n$ $(0,1)$-matrix $E_{ij}$ with all 0's except for a   1 in position $(i,j)$ is in ${\mathcal S}_{m,n}$  and so these $mn$ matrices are linearly independent. Thus ${\mathcal P}_{m,n}$ contains the standard simplex in $M_{m,n}$. In \cite{SS} the following theorem is proved:
 \begin{theorem}\label{th:extreme}
 The set of extreme points of ${\mathcal P}_{m,n}$ is ${\mathcal S}_{m,n}$.
 \end{theorem}
 This theorem admits a simple proof based on the proof in \cite{BD1} that the extreme points of the convex hull of the $n\times n$ ASMs are precisely the $n\times n$ ASMs.  We formulate the following lemma which is essentially the proof  given in \cite{BD1}.
 
 \begin{lemma}\label{lem:BD1}
 Let $\mc{X}_n$ be the set of nonzero vectors
 $x=(x_1,x_2,\ldots,x_n)$  in which, ignoring $0$'s,  the $1$'s and $-1$'s alternate, and the first nonzero  is $1$. Then each vector in $\mc{X}_n$ is an extreme point of the convex hull of $\mc{X}_n$.
 \end{lemma}
 \begin{proof}
 Suppose that $x \in \mc{X}_n$ and 
 \[
 x=\lambda_1 x^{(1)} + \lambda_2 x^{(2)}+\cdots +\lambda_k x^{(k)}
 \]
 where $x^{(i)} \in \mc{X}_n$, $\lambda_i>0$ $(1\le i\le k)$ and $\sum_{i=1}^k\lambda_i=1$.
 If $x$ is a unit vector then, since the $\pm 1$'s in the $x^{(i)}$ alternate, then $x^{(j)}=x$ for some $j$ with $\lambda_j=1$. Now suppose that $x$ is not a unit vector and that $k\ge 2$. Then $x$ contains both a 1 and a $-1$, so  there exists $p$ and $q$ such that $p+1< q$ and $x_p=1$, $x_{p+1}=\cdots=x_{q-1}=0, x_q=-1$. It follows that all $x^{(i)}$ have a 1 in position $p$ and a $-1$ in position $q$. But then all $x^{(i)}$ 
 have either $0$ or $-1$ in position $p+1$ with at least one $-1$ and hence position $p+1$ of $\lambda_1 x^{(1)} + \lambda_2 x^{(2)}+\cdots +\lambda_k x^{(k)}$ does not equal 0, a contradiction.
 \end{proof}
 
 Theorem \ref{th:extreme} follows immediately from Lemma \ref{lem:BD1} by considering any nonzero column of a matrix in ${\mathcal S}_{m,n}$.
 
 In  \cite{BD1} the following notion was introduced. 
 Let $B=[b_{ij}]$ be an $n \times n$ nonnegative matrix. An $n \times n$ matrix $A=[a_{ij}]$ is {\em sum-majorized} by $B$ if  
 \begin{equation}
      \label{eq:B-maj_0}      
      \begin{array}{rl}
       0 \le \sum_{j'=1}^j a_{ij'} \le b_{ij}   &(1 \le i,j \le n), \\*[1.5\smallskipamount]
       0 \le   \sum_{i'=1}^i a_{i'j}\le b_{ij}   &(1 \le i,j \le n),\\*[1.5\smallskipamount]
            \sum_{j=1}^n a_{ij} = b_{in}  &(1 \le i \le n),\\*[1.5\smallskipamount]
            \sum_{i=1}^n a_{ij} = b_{nj} &(1 \le j \le n).
      \end{array}
\end{equation}
Letting $B=J$, the all ones matrix, we  see that an integral matrix $A$  is sum-majorized by $J$ if and only if $A$ is an ASM. Another special case is $B=rJ$, for some positive integer $r$, and this corresponds to the notion of higher spin ASMs that was studied in  \cite{BN}. The following polyhedral result was shown in  \cite{BD1}.

\begin{theorem} 
 \label{thm:k-ASM}
   Let $B=[b_{ij}]$ be an $n\times n$ nonnegative matrix. The convex hull of all integral matrices that are sum-majorized 
by $B$ equals the set  of real matrices $A=[a_{ij}]$ satisfying the linear system in $(\ref {eq:B-maj_0} )$.
\end{theorem}

Now, we connect this to SRMs, and consider the following variation of (\ref{eq:B-maj_0}) for a given $m \times n$ matrix $A=[a_{ij}]$ and a nonnegative integer $c$
 \begin{equation}
      \label{eq:B-maj}      
      \begin{array}{rl}
       0 \le \sum_{l=1}^j a_{il}    &(1 \le i,j \le n), \\*[1.5\smallskipamount]
       0 \le   \sum_{k=1}^i a_{kj}\le 1   &(1 \le i,j \le n),\\*[1.5\smallskipamount]
        \sum_{j=1}^n a_{ij} \le  c  &(1 \le i \le n).
      \end{array}
\end{equation}
We call an integral matrix $A$ satisfying (\ref{eq:B-maj}) a {\em $c$-SRM}. Such a matrix must be a $(0, \pm 1)$-matrix with its nonzeros alternating in every column. When $c \ge n$, a $c$-SRM is precisely an SRM (since an SRM has each row sum at most $n$, and then the third set of constrains in  (\ref{eq:B-maj}) are redundant). 
In general, the parameter $c$ bounds the row sums of the matrix. Let $\mc{S}^c_{m,n}$ denote the class of $c$-SRMs of size $m \times n$, and let the {\em $c$-SRM polytope} $\mc{P}^c_{m,n}$ be defined as the convex hull of the matrices in $\mc{S}^c_{m,n}$. So, when $c \ge n$, we have 
 $\mc{S}^c_{m,n}=\mc{S}_{m,n}$ and $\mc{P}^c_{m,n}=\mc{P}_{m,n}$.
 
 The following result generalizes the linear inequality description of $\mc{P}_{m,n}$  given in  \cite{SS}, and the proof is different and short.
 
\begin{theorem} 
 \label{thm:c-signpolytope}
   Let $c$ be a nonnegative integer. Then  the  polytope $\mc{P}^c_{m,n}$ is equal to the set of  real matrices $A=[a_{ij}]$ satisfying $(\ref{eq:B-maj})$.
\end{theorem}
\begin{proof}
 The proof is a slight variation of our proof of Theorem \ref{thm:k-ASM} in  \cite{BD1}. 
Let $D=(V,E)$ denote the directed graph with vertices $v_{ij}$ ($1 \le i \le m$, $1 \le j \le n$) and arcs $(v_{ij},v_{i+1,j})$ and $(v_{ij},v_{i,j+1})$ for all $i,j$ where the indices are defined. Thus, the vertices correspond to the positions of an $m \times n$ matrix, and arcs from a position go to the neighbor below or to the right. 
 Let $\mc{P}^* \subseteq M_{m,n}$ be the polyhedron consisting of all {\em real} matrices $A=[a_{ij}] \in M_{m,n}$ satisfying the linear system in $(\ref{eq:B-maj})$.  The map $T: A \rightarrow \Sigma(A)$, where $\Sigma(A)$ denotes the sum-matrix,  is an isomorphism on $M_{m,n}$, and therefore $\mc{P}^*$ and its image 
$\Sigma(\mc{P}^*)=\{\Sigma(A): A \in \mc{P}^*\}$
 are isomorphic.   Let $A=[a_{ij}]$ and $S=[s_{ij}]=T(A)$. Then $A=T^{-1}(S)$ is given by 
 \[
       a_{ij}=s_{ij}+s_{i-1,j-1}-s_{i-1,j}-s_{i,j-1} \;\;\;(1 \le i,j \le n)
 \]
where we define $s_{0j}=0$ ($1 \le j \le n$) and $s_{i0}=0$ ($1 \le i \le n$). Moreover 
 \begin{equation}
  \label{eq:A-S-rel}
   \begin{array}{cl}
    \sum_{l=1}^j a_{il}=s_{ij}-s_{i-1,j} &(1\le i,j \le n), \\*[1.5\smallskipamount]
    \sum_{k=1}^i a_{kj}=s_{ij}-s_{i,j-1} &(1\le i,j \le n), \\*[1.5\smallskipamount]
   \sum_{1\le k\le i, 1\le j\le n} a_{kj} = s_{in}  &(1 \le i \le n), \\*[1.5\smallskipamount]
   \sum_{1 \le i\le n, 1 \le l\le j} a_{il} = s_{nj}  &(1 \le j \le n).  
  \end{array}
 \end{equation}
  Note that this isomorphism $T$ and its inverse $T^{-1}$ preserve integrality, that is, an integral matrix is mapped by $T$ and $T^{-1}$ into an integral matrix.

 Now, we claim that  $\Sigma(\mc{P}^*)$ is the set of matrices $S=[s_{ij}]$ satisfying 
 \begin{equation}
      \label{eq:k-ASM-S}      
      \begin{array}{cl}
       0 \le s_{ij}-s_{i-1,j}    &(1 \le i, j \le n), \\*[1.5\smallskipamount]
       0 \le s_{ij}-s_{i,j-1} \le 1    &(1 \le i, j \le n),\\*[1.5\smallskipamount]
           s_{in}-s_{i-1,n}  \le c  &(1 \le i \le n).
      \end{array}
\end{equation}
If fact, if $A$ satisfies (\ref{eq:B-maj}), then, due to (\ref{eq:A-S-rel}),  $S=\Sigma(A)$ satisfies (\ref{eq:k-ASM-S}). Conversely, assume $S=[s_{ij}]$ satisfies (\ref{eq:k-ASM-S}) and let $A=T^{-1}(S)$. Then $A=[a_{ij}]$ satisfies $T(A)=S$, so due to (\ref{eq:A-S-rel}), $A$ satisfies (\ref{eq:B-maj}), as claimed.

 The coefficient matrix of the linear system in (\ref{eq:k-ASM-S}) is totally unimodular, in fact, it is arc-vertex incidence matrix of the directed graph $D$ introduced above (with some repeated arcs/columns). Moreover, all the constants in the system  are integers as $c$ is integral. A standard result from polyhedral theory (see \cite{Schrijver1986}) then implies that $\Sigma(\mc{P}^*)$ is an integral polyhedron, so all extreme points are integral. From the properties of the isomorphism, $\mc{P}^*$ is integral, and this shows the theorem.
\end{proof}

 %An inequality description of ${\mathcal P}(m,n)$ is given in \cite{SS}.
 %\begin{theorem}\label{th:inequ}
 %${\mathcal P}(m,n)$ consists of all real $m\times n$ matrices $X=[x_{ij}]$ satisfying the inequalities
 %\begin{equation}\label{eq:inequ1}
 %0\le \sum_{k=1}^i x_{kj}\le 1\quad (1\le i\le m, 1\le j\le n),
% \end{equation}
% and
% \begin{equation}\label{eq:inequ2}
 %0\le \sum_{l=1}^jx_{il} \quad (1\le i\le m, 1\le j\le n).
% \end{equation}
% \end{theorem}

\medskip

 \begin{theorem}\label{th:decomp}
Let $A$ be an $m\times n$ SRM. Then there exist disjoint subpermutation matrices $P_1, P_2,\ldots,P_N$ and
$\lambda_i\in \{1,-1\}$ $(1\le i\le N)$
such that
\[A=\lambda_1P_1+\lambda_2P_2+\cdots+\lambda_NP_N.\]
\end{theorem}
\begin{proof}
As already remarked, we may assume that all column sums are equal. If $m<n$, then we can include $n-m$ zero rows on the bottom of $A$ and this keeps all column sums equal to 1.
If $m>n$, then we can include $m-n$ columns on the right of $A$ each with a single 1 and this keeps all column sums equal to 1. Thus we may also assume that $m=n$, that is, that $A$ is a square matrix with all column sums equal to 1 and hence its row sum vector $R=(r_1,r_2,\ldots,r_m)$ satisfies $\sum_{i=1}^n r_i=n$. Let $p$ be the maximum row sum of $A$. Then we attach to $A$ on the right an $n\times n(p-1)$ matrix $A_1$ with exactly one 1 in each column so that all row sums now equal $p$. (Note that the arithmetic is correct here: to get all row sums equal to $p$ we need to attach $np-\sum_{i=1}^nr_i= np-n=n(p-1)$ columns with a single 1.) We may attach on the bottom of $A$ an
$n(p-1)\times n$ matrix $A_2$ with $(p-1)$ 1's in each column and one 1 in each row in order to make each column sum equal to $p$.
Let $A_3$ be a $n(p-1)\times n(p-1)$ $(0,1)$-matrix with $(p-1)$ $1$'s in each row and column.
Then the matrix
\[B=\left[\begin{array}{c|c}
A&A_1\\ \hline
A_2&A_3\end{array}\right]\]
is an $np\times np$ $(0,\pm 1)$-matrix, whose row and column sums all equal $p$ and whose only $-1$'s are in $A$, and hence $B+J_{np}$ is a $(0,1,2)$-matrix with all row and column sums equal to $p+1$. Hence $B+J_{np}$ is a sum of $(p+1)$ permutation matrices, and since $J_{np}$ is a sum of permutation matrices, $B$ is a sum of permutation matrices and the negatives of permutation matrices. Restricting this sum to $A$ completes the proof.
\end{proof}

 %\begin{question}{\rm
 %\begin{itemize}
% \item[\rm (i)] In \cite{SS} there seems to be no special identification of the edges of ${\mathcal P}(m,n)$ among its faces, and the corresponding graphs. 
 %\item[\rm (ii)] One can define a Bruhat order on ${\mathcal S}(m,n)$ in the usual way using the sum matrix. 
 %\item[\rm (iii)] How can one generate all sign matrices in ${\mathcal S}(m,n)$? Is there an analogue of ``interchange''.
 %\item[\rm (iv)] Assume that $m=n$. Then ${\mathcal S}(m,n)$ contains the set ${\mathcal A}_n$ of $n\times n$ ASMs and, in particular, the set ${\mathcal P}_n$ of $n\times n$ permutation matrices. Is there something special in this case?\end{itemize}
 %}\hfill{$\Box$}\end{question}

\section{Coda}

In this final section we discuss a connection with an unsolved problem concerning disjoint realization of $(0,1)$-matrices with specified row and column sums.

Consider the class  $ \mc{A}(R,S)$ of $m\times n$ $(0,1)$-matrices where $R=R_1+R_2$ and $S=S_1+S_2$, and  $R_1,R_2,S_1,S_2$ are nonnegative integral vectors. 
If there are matrices  $B_1\in {\mathcal A}(R_1,S_1)$ and $B_2\in {\mathcal A}(R_2,S_2)$ such that $B=B_1+B_2 $ is a matrix in  $ \mc{A}(R,S)$, then $ \mc{A}(R,S)$ has a {\it $(R_1,S_1;R_2,S_2)$ joint realization} and $(B_1,B_2)$ is a {\it joint realization} of $B$ (and of  $ \mc{A}(R,S)$); see e.g.  pages 188--190 in \cite{BR91}. For a joint realization the  matrices $B_1$ and $B_2$ cannot have 1's in common positions and we denote this by $B_1\sqcap B_2=\emptyset$. 

Let $A\in {\mathcal A}^{\pm}(R,S)$. Then $A$ can be uniquely expressed  in the form $A=A_1-A_2$ where $A_1$ and $A_2$ are $(0,1)$-matrices such that $A_1\sqcap A_2=\emptyset$. Let the row and column sum vectors of $A_1$ and $A_2$ be, respectively, $R_1,S_1$ and $R_2,S_2$, and let $R'=R_1+R_2$ and $S'=S_1+S_2$. Then $A'=A_1+A_2$ is an $(R_1,S_1;R_2,S_2)$ joint realization  of ${\mathcal A}(R',S')$ and $(A_1,A_2)$ is a joint  realization of $A'$. Thus, {\it every matrix in ${\mathcal A}^{\pm} (R,S)$ with at least one $1$ and at least one $-1$ gives some   joint realization of $\mc{A}(R,S)$, and every joint realization of $R,S$ gives a matrix in ${\mathcal A}^{\pm} (R,S)$}.

Given  $R_1,S_1$ and $R_2,S_2$ such that both ${\mathcal A}(R_1,S_1)$ and ${\mathcal A}(R_1,S_1)$ are nonempty, it is an unsolved problem to determine whether or not ${\mathcal A}(R_1+R_2,S_1+S_2)$ has an $(R_1,S_1;R_2,S_2)$ joint realization. A necessary condition is that  ${\mathcal A}(R_1+R_2,S_1+S_2)$ is nonempty, but this is not sufficient in general. The following sufficient condition is due to Anstee as a generalization of a theorem of Brualdi and Ross (see Theorem 4.4.14 of \cite{RAB}).

\begin{theorem}\label{th:joint}
Let $R=(r_1,r_2,\ldots,r_m)$ and $S=(s_1,s_2,\ldots,s_n)$ be nonnegative integral vectors.
Let $R_1=(k_1,k_2,\ldots,k_m)$ where for some nonnegative integer $k$, $k\le k_i\le k+1$ for $i=1,2,\ldots,m$,  and let $S_1=(s'_1,s'_2,\ldots,s'_n)$. Let $R_2=R-R_1$ and $S_2=S-S_1$. Then $\mc{A}(R,S)$ has an $(R_1,S_1;R_2,S_2)$ joint realization if and only if both $\mc{A}(R,S)\ne\emptyset$ and $\mc{A}(R_2,S_2)\ne \emptyset$.
\end{theorem}

An immediate corollary of this theorem is the following.

\begin{corollary}\label{cor:joint}
Let $R=(r_1,r_2,\ldots,r_n)$ and $S=(s_1,s_2,\ldots,s_n)$ be nonnegative integral vectors. Let $e=(1,1,\ldots,1)$ where there are $n$ $1$'s. Let $R'
=R-e$ and $S'=S-e$. 
Then there exists a matrix $A=A_1-A_2$ in $\mc{A}^{\pm}(e,e)$ where $A_1\in\mc{A}(R,S)$ if and only if
 $\mc{A}(R,S)\ne\emptyset$ and $\mc{A}(R',S')\ne\emptyset$.
\end{corollary}

\bigskip
{\bf Acknowledgment.} The authors  thank a referee for giving a number of useful comments and suggestions that improved the paper.


\begin{thebibliography}{99}
 
 \bibitem{Anstee83} R.P.~Anstee, The network flows approach for matrices with given row and column sums, {\it Discrete Math.}, 44 (1983), 125--138.
 
 \bibitem{JA} J.C.~Aval, 2007/10. Keys and alternating sign matrices, {\it S\'em. Lothar, Combin.} 59, Art. B59F, 13.
 
 \bibitem{BR91} R.A.~Brualdi, H.J.~Ryser,
\newblock {\em Combinatorial Matrix Theory},
\newblock Cambridge University Press, Cambridge, 1991.

\bibitem{RAB} R.A.~Brualdi, {\it Combinatorial Matrix Classes}, Cambridge University Press, Cambridge, 2006.

\bibitem{BN} R.E.~Behrend,  V.A.~Knight, Higher spin alternating sign matrices,  Electron. J. Combin. 14 (2007), $\#$1.


 \bibitem{BD1} R.A.~Brualdi,  G.~Dahl, Alternating sign matrices, extensions and related cones, {\it Adv. in  Appl. Math.}, 86 (2017), 19--49.
 
  \bibitem{BD2} R.A.~Brualdi,  G.~Dahl, Alternating sign matrices and hypermatrices, and a generalization of Latin squares, {\it Adv. in  Appl. Math.}, 95 (2018), 116--151.
  
 \bibitem{DB} D.~Bressoud, {\it Proofs and Confirmations. The Story of the Alternating Sign Matrix Conjecture}, MAA Spectrum, {\it Math. Assoc. America}, Washington, DC, Cambridge Univ. Press, 1994.
 
 \bibitem{DP90} B.A.~Davey,  H.A.~ Priestly, {\it  Introduction to Lattices and Order}, Cambridge Univ. Press, 1990.
 
 \bibitem{F} M.~Fortin, The MacNeille completion of of the poset of partial injective functions, {\it Electron. J. Combin.}  15 (2008), $\#$R62.
 
 \bibitem{LS} A.~Lascoux, M.-P.~Sch\H utzenberger, Treillis et bases des groupes de Coxeter, {\it Electron. J. Combin.}  3 (1996), $\#$R27.


 \bibitem{MO} A.W.~Marshall, I.~Olkin, B.C.~Arnold, {\it Inequalities: Theory of Majorization and Its Applications}. Second edition. Springer Series in Statistics. Springer, New York, 2011. xxviii+909 pp. 
 
 \bibitem{SS} S.~Solhjem,  J.~Striker, Sign matrix polytopes from Young tableaux, {\it Linear Algebra Appl.}, 574 (2019), 84--122.
 
 \bibitem{St}   J.~Striker, The alternating sign matrix polytope. {\it Electron. J. Combin.}, 16 (2009), no. 1, Research paper 41, 15 pp.
 
\bibitem{Schrijver1986}  
A.~Schrijver, 
\newblock  Theory of Linear and Integer Programming,
\newblock Wiley-Interscience, Chichester, 1986.

%\bibitem{PT} P.~Terwilliger, A poset $\Phi_n$  whose maximal chains are in bijection with %the $n \times n$ alternating sign matrices, {\it Linear Algebra Appl.}, 554 (2018), 79--85.

 \end{thebibliography}
 \end{document}